\documentclass[a4paper,12pt]{article}
\UseRawInputEncoding
\usepackage{amsmath}
\usepackage{ascmac}
\usepackage{amssymb}
\usepackage{multicol}
\usepackage{graphicx}
\usepackage[dvipdfmx]{color}
\usepackage{amscd}
\usepackage{url}
\usepackage{comment}
\usepackage{bm}
\usepackage{theorem}
\usepackage{latexsym}
\usepackage{mathrsfs}
\usepackage{framed,color}
\usepackage{fancybox}
\usepackage{ulem}
\usepackage[dvipdfmx]{hyperref}
\usepackage[all]{xy}
\theorembodyfont{\normalfont}
\newtheorem{theorem}{Theorem}[section]
\newtheorem{definition}[theorem]{Definition}
\newtheorem{proposition}[theorem]{Proposition}

\newtheorem{remark}[theorem]{Remark}
\newtheorem{example}[theorem]{Example}

\definecolor{shadecolor}{gray}{0.85}
\makeatletter
\newenvironment{proof}[1][\proofname]{\par
  \normalfont
  \topsep6\p@\@plus6\p@ \trivlist
  \item[\hskip\labelsep{\bfseries #1}\@addpunct{\bfseries}]\ignorespaces
}{
  \endtrivlist
}
\newcommand{\proofname}{\textit{Proof.}}

\newcommand{\ctext}[1]{\raise0.2ex\hbox{\textcircled{\scriptsize{#1}}}}

	\@addtoreset{equation}{section}

	\@addtoreset{figure}{section}

	\@addtoreset{table}{section}
	
\makeatother
\def\qed{\hfill$\Box$}
\title{On Gauss-Bonnet type formulas for mappings between surfaces with boundary and their applications}
\author{Kyoya Hashibori\footnote{hashibori.kyoya.a7@elms.hokudai.ac.jp}}
\date{Department of Mathematics, Graduate School of Science, Hokkaido University, Kita-10 Nishi-8, Kita-ku, Sapporo 060-0810, JAPAN}
\begin{document}

\maketitle

\begin{abstract}

We define singular points of the first kind and singular points of the second kind as singular points of mappings between surfaces. Typical examples of these singular points are fold singular points and cusp singular points, respectively. We can construct coherent tangent bundles, which are natural intrinsic formulations of wave fronts, by using mappings between surfaces. It is known that two types of Gauss-Bonnet type formulas hold for coherent tangent bundles over surfaces (possibly with boundary). Hence, we obtain two Gauss-Bonnet type formulas for mappings between surfaces (possibly with boundary). By applying these Gauss-Bonnet type formulas, we prove the Levine formula that relates the rotation indices to the Euler characteristic, and the Quine-Fukuda-Ishikawa formula that relates the mapping degree to the Euler characteristic and the number of singular points.
\end{abstract}

\section{Introduction}

The classical Gauss-Bonnet theorem claims that the integral of the Gaussian curvature of a regular surface $M$ is equal to the Euler characteristic $\chi(M)$:
\begin{equation}
\int_MKdA=2\pi\chi(M),\label{1.1}
\end{equation}
where $dA$ is the area form on $M$.

Saji, Umehara, and Yamada used coherent tangent bundles, natural intrinsic formulations of wave fronts, to generalize $(\ref{1.1})$ to wave fronts, which are surfaces that admit singular points (see \cite{6,7,8,9}, Definition $\ref{def2.9}$). Later, Domitrz and Zwierzy\'{n}ski generalized the theory of coherent tangent bundles to the case of surfaces with boundary, and they generalized the Gauss-Bonnet type formulas derived by Saji, Umehara, and Yamada to the case of surfaces with boundary (see \cite{2}).

Since coherent tangent bundles are intrinsic formulations of wave fronts, there are examples other than wave fronts that induce coherent tangent bundles, one of which is a mapping between surfaces (see \cite{8}, Example $\ref{ex2.10}$). Therefore, the Gauss-Bonnet type formulas are derived for mappings between surfaces. The following assertion is the Gauss-Bonnet type formulas for mappings between surfaces:

\begin{theorem}[Gauss-Bonnet type formulas]
\label{thm2.18}

{\it Let $M$ be a compact oriented surface with boundary, $(N,g)$ an oriented Riemannian surface (possibly with boundary), and $f:M\to N$ a $C^\infty$-map. We suppose that $f$ admits only singular points of the first kind and admissible singular points of the second kind, and that the set of singular points $\Sigma$ is transversal to the boundary $\partial M$. Then, the following two equations hold:
\begin{itemize}
\item[$(1)$]$\displaystyle{\int_M(K_N\circ f)dA+2\int_{\Sigma}\kappa_sds+\int_{\partial{M}}\kappa_gds=2\pi\chi(M)+\sum_{p\in(\Sigma\cap\partial{M})^{\mathrm{null}}}\left(2\alpha^+(p)-\pi\right)}$,
\item[$(2)$]$\displaystyle{\int_{M}(K_N\circ f)d\widehat{A}+\int_{\partial M\cap M^+}\kappa_gds-\int_{\partial M\cap M^-}\kappa_gds}$

$\displaystyle{=2\pi\left(\chi(M^+)-\chi(M^-)\right)+2\pi\left(\#S^+-\#S^-\right)+\pi\left(\#(\Sigma\cap\partial{M})^+-\#(\Sigma\cap\partial{M})^-\right)}$,
\end{itemize}
where $\kappa_gds$ is a geodesic curvature measure, $S^+$ (resp. $S^-$) denotes the set of the positive (resp. negative) singular points of the second kind in $M\backslash\partial M$, and $(\Sigma\cap\partial{M})^+$ (resp. $(\Sigma\cap\partial{M})^{\mathrm{null}}$, $(\Sigma\cap\partial{M})^-$) denotes the set of the positive (resp. null, negative) singular points on $\partial M$. Other notations are explained in $\S\ref{sec2}$.}
\end{theorem}

Saji, Umehara, and Yamada derived the Levine formula related to the rotation indices and the Quine formula related to the mapping degree by applying the Gauss-Bonnet type formulas for surfaces without boundary (see \cite{5,10,8}). Domitrz and Zwierzy\'{n}ski showed the ``special case'' of the Fukuda-Ishikawa formula, which is a generalization of the Quine formula to surfaces with boundary (see \cite{2,3}). (The ``special case'' shown by Domitrz and Zwierzy\'{n}ski is the case where the singular points on the boundary of a surface are null (see Definition $\ref{def2.17}$).)

Based on these previous works, we give a proof of Theorem $\ref{thm2.18}$, and as an application of this theorem, we give a generalization of the Levine formula to the case of surfaces with boundary and a generalization of the Quine-Fukuda-Ishikawa formula (see Theorem $\ref{thm3.1}$, $\ref{thm4.1}$).

The paper consists of the following: In \S$\ref{sec2}$, we briefly review the theory of non-degenerate singular points of mappings between surfaces and the theory of coherent tangent bundles induced by mappings between surfaces, and give a proof of Theorem $\ref{thm2.18}$. In \S$\ref{sec3}$, we give a generalization of the Levine formula to the case of surfaces with boundary by using Theorem $\ref{thm2.18}$ $(1)$. In \S$\ref{sec4}$, we give a generalization of the Quine-Fukuda-Ishikawa formula by using Theorem $\ref{thm2.18}$ $(2)$.

\section{Mappings between surfaces and Gauss-Bonnet type formulas}\label{sec2}

Let $M,N$ be oriented surfaces (possibly with boundary) and $f:M\to N$ a $C^\infty$-map.

\begin{definition}
\label{def2.1}

A point $p\in M$ is a \textit{singular point} of $f$ if the rank of the derivative $df_p:T_pM\to T_{f(p)}N$ of $f$ at $p$ is less than $2$. The set of the singular points of $f$ is denoted by $\Sigma$. A point which is not a singular point is called \textit{regular point}.
\end{definition}

We take a local coordinate system $(V;x,y)$ compatible with the orientation of $N$. We take a positive orthonormal frame field $\left\{\bm{e}_1,\bm{e}_2\right\}$ on the restriction $TN|_V$ of the tangent bundle $TN$ to $V$ and let $\left\{\omega_1,\omega_2\right\}$ be its dual basis field.

\begin{definition}
\label{def2.2}

We define a $2$-form $dA_N$ on $V$ as
\begin{equation*}
dA_N:=\omega_1\wedge\omega_2.\label{2.1}
\end{equation*}
Then, we can check that $dA_N$ does not depend on the choice of the positive orthonormal frame fields on $TN|_V$, so $dA_N$ defines a globally defined $2$-form on $N$. We call $dA_N$ the \textit{area form} on $N$, and the pull-back $d\widehat{A}$ of $dA_N$ by $f$ is called the \textit{signed area form} on $M$:
\begin{equation}
d\widehat{A}:=f^*(dA_N).\label{2.2}
\end{equation}
\end{definition}

We take a local coordinate system $(U;u,v)$ compatible with the orientation of $M$ satisfying $f(U)\subset V$.

\begin{definition}
\label{def2.3}

We define a $C^\infty$-function $\lambda$ on $U$ as
\begin{eqnarray}
\lambda&:=&d\widehat{A}\left(\frac{\partial}{\partial u},\frac{\partial}{\partial v}\right)=(dA_N)_f\left(f_u,f_v\right),\label{2.3}
\end{eqnarray}
where $f_u:=df\left(\frac{\partial}{\partial u}\right),\ f_v:=df\left(\frac{\partial}{\partial v}\right)$. We call $\lambda$ the \textit{signed area density function}.
\end{definition}

By $(\ref{2.3})$, $d\widehat{A}$ is expressed on $U$ as
\begin{equation*}
d\widehat{A}=\lambda du\wedge dv.\label{2.4}
\end{equation*}
Hence, $d\widehat{A}$ is a smooth $2$-form on $M$. Furthermore, the set of the singular points $\Sigma$ on $U$ is represented as
\begin{equation*}
\Sigma\cap U=\left\{p\in U\mid\lambda(p)=0\right\}.\label{2.5}
\end{equation*}

\begin{definition}
\label{def2.4}

We define a continuous $2$-form $dA$ on $U$ as
\begin{equation*}
dA:=|\lambda|du\wedge dv.\label{2.6}
\end{equation*}
Then, we can check that $dA$ does not depend on the choice of the local coordinate systems compatible with the orientation of $M$, so $dA$ defines a globally defined continuous $2$-form on $M$. We call $dA$ the \textit{area form} on $M$.
\end{definition}

Using the $2$-forms $d\widehat{A}$ and $dA$ on $M$, we define sub-surfaces $M^+$ and $M^-$ of $M$ as
\begin{eqnarray*}
M^+:=\left\{p\in M\backslash\Sigma\mid dA_p=d\widehat{A}_p\right\},\ M^-:=\left\{p\in M\backslash\Sigma\mid dA_p=-d\widehat{A}_p\right\},\label{2.7}
\end{eqnarray*}
respectively. We note that $M^+$ and $M^-$ are represented on $U$ as
\begin{equation}
M^+\cap U=\left\{p\in U\mid \lambda(p)>0\right\},\ M^-\cap U=\left\{p\in U\mid\lambda(p)<0\right\},\label{2.8}
\end{equation}
respectively. In particular, by $(\ref{2.8})$, we see that $M^+$ (resp. $M^-$) is the set of the regular points such that the derivative $df$ of $f$ is orientation preserving (resp. orientation inverting).

\begin{definition}
\label{def2.5}

A singular point $p\in M$ is \textit{non-degenerate} if the exterior derivative $d\lambda$ of the signed area density function $\lambda$ does not vanish at $p$.
\end{definition}

\begin{definition}
\label{def2.6}

Let $p\in M$ be a non-degenerate singular point. Then, by the implicit function theorem, there exists a sufficiently small neighborhood $U$ of $p$ such that the set of the singular points $\Sigma$ on $U$ can be parameterized by a regular curve $\gamma(t)\ (\gamma(0)=p)$. Such a curve is called a \textit{singular curve}, and the direction of the tangent vector $\gamma^\prime(t):=\frac{d\gamma}{dt}(t)$ is called the \textit{singular direction}.

On the other hand, we can check that the rank of the derivative $df_p:T_pM\to T_{f(p)}N$ of $f$ at $p$ is $1$. Hence, the dimension of the sub-linear space $\ker{df_p}$ of $T_pM$ is $1$, which is called the \textit{null direction}. Also, since $\gamma(t)$ consists of non-degenerate singular points, the null direction at $\gamma(t)$ is a $1$-dimensional linear space. Therefore, we can take a smooth vector field $\eta(t)$ along $\gamma(t)$ such that $\eta(t)$ points in the null direction at $\gamma(t)$, which is called a \textit{null vector field}.
\end{definition}

\begin{definition}
\label{def2.7}

Let $p\in M$ be a non-degenerate singular point, $\gamma(t)\ (\gamma(0)=p)$ a singular curve through $p$, and $\eta(t)$ a null vector field along $\gamma(t)$. $p$ is \textit{of the first} (resp. \textit{second}) \textit{kind} if the singular direction and the null direction at $p$ are different (resp. the same), i.e.,
\begin{equation*}
\det\left(\gamma^\prime(0),\eta(0)\right)\neq0\ \ \ \left(\mbox{resp}.\ \det\left(\gamma^\prime(0),\eta(0)\right)=0\right),\label{2.9}
\end{equation*}
where $\det$ denotes the determinant function of $2\times2$-matrices obtained by identifying $T_pM$ with $\mathbb{R}^2$. Moreover, a singular point of the second kind $p$ is \textit{admissible} if there exists a neighborhood $V$ of $p$ such that the set of the singular points on $V\backslash\{p\}$ consists only of singular points of the first kind, that is, there exists a non-negative integer $k$ such that
\begin{equation*}
\delta(0)=0,\ \delta^\prime(0)=0,\ \cdots,\ \delta^{(k)}(0)=0,\ \delta^{(k+1)}(0)\neq0,\label{2.10}
\end{equation*}
where $\delta(t):=\det\left(\gamma^\prime(t),\eta(t)\right)$.
\end{definition}

\begin{example}
\label{ex2.8}

A singular point $p$ of a $C^\infty$-map $f:M\to N$ between surfaces $M,N$ is a \textit{fold} (resp. \textit{cusp}) \textit{singular point} if the germ of $f$ at $p$ is right-left equivalent to the germ of a $C^\infty$-map defined by
\begin{equation*}
(u,v)\mapsto\left(u,v^2\right)\ \ \ \left(\mbox{resp}.\ (u,v)\mapsto\left(u^3-3uv,v\right)\right)\label{2.11}
\end{equation*}
at the origin $(0,0)$. Here, two map-germs
\begin{equation*}
g:(M,p)\to(N,g(p)),\ h:(M,q)\to(N,h(q))\label{2.12}
\end{equation*}
are \textit{right-left equivalent} if there exist two diffeomorphism-germs
\begin{equation*}
\varphi:(M,p)\to(M,q),\ \psi:(N,g(p))\to(N,h(q))\label{2.13}
\end{equation*}
such that
\begin{equation*}
\psi\circ g=h\circ\varphi\label{2.14}
\end{equation*}
holds. Then, we can check that $p$ is a singular point of the first kind (resp. an admissible singular point of the second kind).
\end{example}

In order to describe the Gauss-Bonnet type formulas for mappings between surfaces, we define a coherent tangent bundle over $M$.

\begin{definition}
\label{def2.9}

A \textit{coherent tangent bundle} over $M$ is a $5$-tuple $(M,\mathcal{E},\langle\cdot,\cdot\rangle,D,\varphi)$ such that it satisfies the following conditions:
\begin{itemize}
\item[$(1)$]$\mathcal{E}$ is an orientable vector bundle of rank $2$ on $M$, $\langle\cdot,\cdot\rangle$ is a metric on $\mathcal{E}$, and $D$ is a metric connection on $(\mathcal{E},\langle\cdot,\cdot\rangle)$,
\item[$(2)$]$\varphi:TM\to\mathcal{E}$ is a bundle homomorphism such that, for any vector fields $X,Y$ on $M$,
\begin{equation*}
D_X\varphi(Y)-D_Y\varphi(X)=\varphi([X,Y])\label{2.15}
\end{equation*}
holds.
\end{itemize}
\end{definition}

\begin{example}[\mbox{\cite[Example $2.3$]{8}}]
\label{ex2.10}

Let $(N,g)$ be an orientable Riemannian surface (possibly with boundary) and $f^*TN$ the pull-back of the tangent bundle $TN$ by a $C^\infty$-map $f:M\to N$. Then, the pull-back $\langle\cdot,\cdot\rangle:=f^*g$ of the metric $g$ by $f$ is a metric on $f^*TN$ and the restriction $D$ of the Levi-Civita connection on $N$ to $f^*TN$ is a metric connection on $(f^*TN,\langle\cdot,\cdot\rangle)$. Moreover, the derivative $df:TM\to f^*TN$ of $f$ is a bundle homomorphism such that, for any vector fields $X,Y$ on $M$,
\begin{equation*}
D_Xdf(Y)-D_Ydf(X)=df([X,Y])\label{2.16}
\end{equation*}
holds. Hence, the $5$-tuple $(M,f^*TN,\langle\cdot,\cdot\rangle,D,df)$ is a coherent tangent bundle over $M$, which is called the \textit{coherent tangent bundle over $M$ induced by $f$}.
\end{example}

We equip $N$ with a Riemannian metric. We fix a coherent tangent bundle $(M,f^*TN,\langle\cdot,\cdot\rangle,D,df)$ induced by $f$.

\begin{definition}
\label{def2.11}

We define the \textit{first fundamental form} $ds^2$ on $M$ as the pull-back of the metric $\langle\cdot,\cdot\rangle$ by $f$:
\begin{equation*}
ds^2:=f^*\langle\cdot,\cdot\rangle.\label{2.17}
\end{equation*}
A point $p\in M$ is a \textit{singular point} of $ds^2$ if $ds^2$ is degenerate at $p$. We note that singular points of $ds^2$ are singular points of $f$.
\end{definition}

\begin{definition}
\label{def2.12}

By the definition of the connection $D$ on $f^*TN$, for any vector field $X$ on $f^{-1}(V)$,
\begin{equation*}
\langle D_X(\bm{e}_i\circ f),\bm{e}_i\rangle=g\left(\nabla_{df(X)}\bm{e}_i,\bm{e}_i\right)=0\ \ \ (i=1,2)\label{2.18}
\end{equation*}
holds, where $\nabla$ is the Levi-Civita connection on $N$. Hence, there exists a unique $1$-form $\mu$ on $f^{-1}(V)$ such that, for any vector field $X$ on $f^{-1}(V)$, it satisfies
\begin{equation}
D_X(\bm{e}_1\circ f)=-\mu(X)\bm{e}_2,\ D_X(\bm{e}_2\circ f)=\mu(X)\bm{e}_1,\label{2.19}
\end{equation}
which is called the \textit{connection form} with respect to $D$.
\end{definition}

Since we can check that the exterior derivative $d\mu$ of the connection form $\mu$ does not depend on the choice of the positive orthonormal frame fields on $TN$, so $d\mu$ defines a globally defined $2$-form on $M$. Furthermore, we can check that
\begin{equation*}
d\mu=Kd\widehat{A}\label{2.20}
\end{equation*}
holds on the set of the regular points $M\backslash\Sigma$, where $K$ is the Gaussian curvature defined on $M\backslash\Sigma$. Hence, by the continuity of $d\mu$, the $2$-form $Kd\widehat{A}$ on $M\backslash\Sigma$ can be smoothly extended to the $2$-form $d\mu$ on $M$.

\begin{proposition}
\label{prop2.13}

{\it Let $K_N$ be the Gaussian curvature of $(N,g)$. Then, 
\begin{equation*}
Kd\widehat{A}=(K_N\circ f)d\widehat{A}\label{2.21}
\end{equation*}
holds on the set of the regular points $M\backslash\Sigma$.}
\end{proposition}

\begin{proof}

Let $\mu_N$ be the connection form with respect to the Levi-Civita connection on $N$. Then, for any vector field $X$ on $M\backslash\Sigma$,
\begin{eqnarray*}
D_X(\bm{e}_1\circ f)&=&-\mu_N(df(X))\bm{e}_2=-f^*\mu_N(X)\bm{e}_2,\label{2.22}\\
D_X(\bm{e}_2\circ f)&=&\mu_N(df(X))\bm{e}_1=f^*\mu_N(X)\bm{e}_1\label{2.23}
\end{eqnarray*}
holds. By the uniqueness of the connection form with respect to the connection $D$ on $f^*TN$, we have
\begin{equation*}
\mu=f^*\mu_N,\label{2.24}
\end{equation*}
where $\mu$ is the connection form with respect to $D$ (see $(\ref{2.19})$). Hence, by noting that the exterior derivative $d\mu_N$ of $\mu_N$ is equal to $K_NdA_N$, we have
\begin{equation*}
Kd\widehat{A}=d\mu=d(f^*\mu_N)=f^*(d\mu_N)=f^*(K_NdA_N)=(K_N\circ f)d\widehat{A}\label{2.25}
\end{equation*}
on $M\backslash\Sigma$.\qed
\end{proof}

Let $p\in M$ be a singular point of the first kind. Then, there exists a singular curve $\gamma(t)\ (\gamma(0)=p)$ passing through $p$. If necessary, we can take a sufficiently small neighborhood $U$ of $p$ so that $U\cap\mathrm{Im}\gamma$ consists only of singular points of the first kind, so we may assume that $\gamma(t)$ consists only of singular points of the first kind. We take a null vector field $\eta(t)$ along $\gamma(t)$ so that $\left\{\gamma^\prime(t),\eta(t)\right\}$ is compatible with the orientation of $M$. Since $\gamma(t)$ is a non-degenerate singular point, the rank of the derivative $df_{\gamma(t)}$ of $f$ at $\gamma(t)$ is $1$. Hence, $f^\prime(\gamma(t))$ never vanishes. On the other hand, since $\lambda(\gamma(t))=0$, we have $d\lambda_{\gamma(t)}(\gamma^\prime(t))=0$. Also, since $\gamma(t)$ is a non-degenerate singular point, the dimension of $\ker d\lambda_{\gamma(t)}$ is $1$. Hence, since $\gamma^\prime(t)$ is linearly independent of $\eta(t)$, $d\lambda_{\gamma(t)}(\eta(t))$ never vanishes.

\begin{definition}
\label{def2.14}

We define the \textit{singular curvature} $\kappa_s(t)$ of a singular curve $\gamma(t)$ consisting of singular points of the first kind as
\begin{eqnarray*}
\kappa_s(t):=\mathrm{sgn}(d\lambda_{\gamma(t)}(\eta(t)))\frac{\langle f^{\prime\prime}(\gamma(t)),\bm{n}(t)\rangle}{|f^\prime(\gamma(t))|^2},\label{2.26}
\end{eqnarray*}
where $\bm{n}(t)$ is a smooth section along $\gamma(t)$ of $f^*TN$ such that $\left\{\frac{f^\prime(\gamma(t))}{|f^\prime(\gamma(t))|},\bm{n}(t)\right\}$ gives a positive orthonormal frame field on $f^*TN$, and $f^{\prime\prime}(\gamma(t))$ is the covariant derivative $D_tf^\prime(\gamma(t))$ of $f^\prime(\gamma(t))$ with respect to $\gamma^\prime(t)$.
\end{definition}

\begin{proposition}[\mbox{\cite[Proposition $1.7$, $2.11$]{6}}]
\label{prop2.15}

{\it The following assertions hold:
\begin{itemize}
\item[$(1)$]The value of the singular curvature $\kappa_s$ is independent of the choice of the parameters of the singular curve $\gamma$, the orientation of $\gamma$ and the orientation of $M,N$,
\item[$(2)$]Let $\gamma(t)$ be a singular curve passing through an admissible singular point of the second kind. Then, the singular curvature measure $\kappa_s(t)ds$ defines a bounded $1$-form along $\gamma(t)$, where $ds:=|\gamma^\prime(t)|dt$ is the arc-length measure of $\gamma(t)$. Hence, the integral of the singular curvature measure is well-defined.
\end{itemize}}
\end{proposition}

Let $M$ be a compact oriented surface with boundary, $(N,g)$ an oriented Riemannian surface (possibly with boundary), and $(M,f^*TN,\langle\cdot,\cdot\rangle,D,df)$ a coherent tangent bundle induced by a $C^\infty$-map $f:M\to N$. We suppose that $f$ admits only singular points of the first kind and admissible singular points of the second kind, and that the set of the singular points $\Sigma$ is transversal to the boundary $\partial M$. We triangulate $M$ so that the singular points of the second kind in the interior $M\backslash\partial M$ and the singular points on $\partial M$ are vertices.

\begin{proposition}[\mbox{\cite[Theorem $\mathrm{A}$]{6}, \cite[Theorem $2.13$]{2}}]
\label{prop2.16}

{\it The following assertions hold.
\begin{itemize}
\item[$(1)$]Let $p\in M\backslash\partial M$ be a singular point of the second kind. Then, the interior angle $\alpha^+(p)$ (resp. $\alpha^-(p)$) on the $M^+$ side (resp. $M^-$ side) at $p$ satisfies
\begin{equation*}
\alpha^+(p)+\alpha^-(p)=2\pi,\ \alpha^+(p)-\alpha^-(p)\in\left\{-2\pi,0,2\pi\right\},\label{2.27}
\end{equation*}
\item[$(2)$]Let $p\in\partial M$ be a singular point. If the null direction at $p$ is different from the direction of $\partial M$, then $\alpha^+(p)$ and $\alpha^-(p)$ satisfy
\begin{equation*}
\alpha^+(p)+\alpha^-(p)=\pi,\ \alpha^+(p)-\alpha^-(p)\in\left\{-\pi,\pi\right\}.\label{2.28}
\end{equation*}
On the other hand, if the null direction at $p$ is the same as the direction of $\partial M$, then $\alpha^+(p)$ and $\alpha^-(p)$ satisfy
\begin{equation*}
\alpha^+(p)-\alpha^-(p)=0.\label{2.29}
\end{equation*}
\end{itemize}}
\end{proposition}

\begin{definition}
\label{def2.17}

A singular point of the second kind $p\in M\backslash\partial M$ is \textit{positive} (resp. \textit{null}, \textit{negative}) if
\begin{eqnarray*}
\alpha^+(p)-\alpha^-(p)=2\pi\ \ \ \left(\mbox{resp}.\ \alpha^+(p)-\alpha^-(p)=0,\ \alpha^+(p)-\alpha^-(p)=-2\pi\right)\label{2.30}
\end{eqnarray*}
holds. Similarly, a singular point $p\in\partial M$ is \textit{positive} (resp. \textit{null}, \textit{negative}) if 
\begin{eqnarray*}
\alpha^+(p)-\alpha^-(p)=\pi\ \ \ \left(\mbox{resp}.\ \alpha^+(p)-\alpha^-(p)=0,\ \alpha^+(p)-\alpha^-(p)=-\pi\right)\label{2.31}
\end{eqnarray*}
holds.
\end{definition}

Finally, we give a proof of Theorem $\ref{thm2.18}$.

\begin{proof}[\textit{Proof of Theorem $\ref{thm2.18}$}.]

By Example $\ref{ex2.10}$, we can construct a coherent tangent bundle $(M,f^*TN,\langle\cdot,\cdot\rangle,D,df)$ induced by $f$. Therefore, by \cite[Theorem $2.20$]{2}, the following two formulas hold:
\begin{eqnarray}
&&\int_MKdA+2\int_{\Sigma}\kappa_sds+\int_{\partial{M}}\kappa_gds\nonumber\\
&=&2\pi\chi(M)+\sum_{p\in(\Sigma\cap\partial{M})^{\mathrm{null}}}\left(2\alpha^+(p)-\pi\right),\label{2.32}\\
&&\int_{M}Kd\widehat{A}+\int_{\partial M\cap M^+}\kappa_gds-\int_{\partial M\cap M^-}\kappa_gds\nonumber\\
&=&2\pi\left(\chi(M^+)-\chi(M^-)\right)+2\pi\left(\#S^+-\#S^-\right)\nonumber\\
&&+\pi\left(\#(\Sigma\cap\partial{M})^+-\#(\Sigma\cap\partial{M})^-\right).\label{2.33}
\end{eqnarray}
Here, by Proposition $\ref{prop2.13}$, we have
\begin{eqnarray}
Kd\widehat{A}&=&(K_N\circ f)d\widehat{A},\label{2.34}\\
KdA&=&(K_N\circ f)dA.\label{2.35}
\end{eqnarray}
Therefore, $(\ref{2.32})$ and $(\ref{2.35})$ yield the formula $(1)$, and $(\ref{2.33})$ and $(\ref{2.34})$ yield the formula $(2)$.\qed
\end{proof}

\begin{remark}
\label{rem2.19}

Theorem $\ref{thm2.18}$ is a generalization of \cite[Proposition $4.2$]{2} in the following sense: The formula given in \cite[Proposition $4.2$]{2} is derived under the assumption that a $C^\infty$-map between surfaces admits only fold singular points and cusp singular points. Fold singular points and cusp singular points are typical examples of singular points of the first kind and admissible singular points of the second kind, respectively (see Example $\ref{ex2.8}$).
\end{remark}

\section{Application of Theorem $\ref{thm2.18}$ $(1)$}\label{sec3}

By applying Theorem $\ref{thm2.18}$ $(1)$, we can show the following assertion.

\begin{theorem}
\label{thm3.1}

{\it Let $M$ be a compact oriented surface with boundary and $f:M\to\mathbb{R}^2$ a $C^\infty$-map from $M$ into the plane $\mathbb{R}^2$. We suppose that $f$ admits only singular points of the first kind and admissible singular points of the second kind, and that the set of the singular points $\Sigma$ does not intersect the boundary $\partial M$. Let $\left\{c_1,\cdots,c_r\right\}$ (resp. $\left\{e_1,\cdots,e_s\right\}$) be the components of $\Sigma$ (resp. $\partial M$). Then, there exists a unique orientation of $\left\{c_1,\cdots,c_r\right\}$ and $\left\{e_1,\cdots,e_s\right\}$ respectively, such that the following equation holds:
\begin{equation}
\frac{\chi(M)}{2}=\sum_{i=1}^rI(c_i)+\frac{1}{2}\sum_{i=1}^sI(e_i),\label{3.1}
\end{equation}
where $I(c_i)$ (resp. $I(e_i)$) is the rotation index of $c_i$ (resp. $e_i$).}
\end{theorem}

\begin{proof}

Equipping $\mathbb{R}^2$ with Euclidean metric, we can construct a coherent tangent bundle $(M,f^*T\mathbb{R}^2,\langle\cdot,\cdot\rangle,D,df)$ induced by $f$ (see Example $\ref{ex2.10}$). By the fact that the Gaussian curvature of $\mathbb{R}^2$ is zero, and the assumption that $\Sigma$ does not intersect $\partial M$, Theorem $\ref{thm2.18}$ $(1)$ is reduced to
\begin{equation}
2\int_{\Sigma}\kappa_sds+\int_{\partial{M}}\kappa_gds=2\pi\chi(M).\label{3.2}
\end{equation}
By \cite[Proposition $5.4.6$]{9}, the curvature of $f(c_i)\ (i=1,\cdots,r)$ as a planar curve equals to the singular curvature of $c_i$. Therefore, we have
\begin{equation}
\int_{\Sigma}\kappa_sds=2\pi\sum_{i=1}^rI(c_i).\label{3.3}
\end{equation}
Furthermore, the curvature of $f(e_i)\ (i=1,\cdots,s)$ as a planar curve is equal to the geodesic curvature of $e_i$. Therefore, we have
\begin{equation}
\int_{\partial M}\kappa_gds=2\pi\sum_{i=1}^sI(e_i).\label{3.4}
\end{equation}
Finally, by substituting $(\ref{3.3})$ and $(\ref{3.4})$ into $(\ref{3.2})$, we obtain the formula $(\ref{3.1})$.\qed
\end{proof}

\begin{remark}
\label{rem3.2}

Theorem $\ref{thm3.1}$ is a generalization of the Levine formula in \cite[p. $2567$]{4} in the following sense: The formula given in \cite[p. $2567$]{4} is derived under the assumption that a $C^\infty$-map between surfaces admits only fold singular points and cusp singular points. Fold singular points and cusp singular points are typical examples of singular points of the first kind and admissible singular points of the second kind, respectively (see Example $\ref{ex2.8}$).
\end{remark}

\section{Application of Theorem $\ref{thm2.18}$ $(2)$}\label{sec4}

Let $M,N$ be compact oriented surfaces without boundary and $f:M\to N$ a $C^\infty$-map. We suppose that $f$ admits only singular points of the first kind and admissible singular points of the second kind.

We take a regular curve on $N$ such that it is transversal to $f$. By regarding this curve as the boundary of $N$, we consider $N$ as a surface with boundary. Then, the inverse $f^{-1}(\partial N)$ of the boundary $\partial N$ is a regular curve on $M$. By regarding this curve as the boundary of $M$, we consider $M$ as a surface with boundary. Hence, $f$ is a $C^\infty$-map satisfying $f^{-1}(\partial N)=\partial M$.

We suppose that the set of the singular points $\Sigma$ of $f$ is transversal to the boundary $\partial M$.

We equip $N$ with a Riemannian metric. Then, we can construct a coherent tangent bundle $(M,f^*TN,\langle\cdot,\cdot\rangle,D,df)$ induced by $f$ (see Example $\ref{ex2.10}$). Hence, Theorem $\ref{thm2.18}$ $(2)$ holds.

By $(\ref{2.2})$, we have
\begin{equation}
\int_M(K_N\circ f)d\widehat{A}=\int_Mf^*(K_NdA_N)=\mathrm{deg}(f)\int_NK_NdA_N,\label{4.1}
\end{equation}
where $\deg(f)$ is the mapping degree of $f$. By the Gauss-Bonnet theorem on $N$, $(\ref{4.1})$ is reduced to
\begin{equation}
\int_M(K_N\circ f)d\widehat{A}=\mathrm{deg}(f)\left(2\pi\chi(N)-\int_{\partial N}\kappa_g^Nds_N\right),\label{4.2}
\end{equation}
where $\kappa_g^Nds_N$ is a geodesic curvature measure on $\partial N$. On the other hand, the property of $f^{-1}(\partial N)=\partial M$ yields
\begin{equation}
\int_{\partial M\cap M^+}\kappa_gds-\int_{\partial M\cap M^-}\kappa_gds=\mathrm{deg}(f)\int_{\partial N}\kappa_g^Nds_N.\label{4.3}
\end{equation}
Hence, by substituting $(\ref{4.2})$ and $(\ref{4.3})$ into Theorem $\ref{thm2.18}$ $(2)$, we obtain
\begin{eqnarray}
\mathrm{deg}(f)\chi(N)&=&\chi(M^+)-\chi(M^-)+\#S^+-\#S^-\nonumber\\
&&+\frac{1}{2}\left(\#(\Sigma\cap\partial{M})^+-\#(\Sigma\cap\partial{M})^-\right).\label{4.4}
\end{eqnarray}
Here, by noting that $\chi(M^\pm)=\chi(\overline{M^\pm})-\chi(\Sigma)$, $(\ref{4.4})$ reduces to
\begin{eqnarray}
\mathrm{deg}(f)\chi(N)&=&\chi(\overline{M^+})-\chi(\overline{M^-})+\#S^+-\#S^-\nonumber\\
&&+\frac{1}{2}\left(\#(\Sigma\cap\partial{M})^+-\#(\Sigma\cap\partial{M})^-\right).\label{4.5}
\end{eqnarray}
Furthermore, since $\Sigma$ is transversal to $\partial M$, we have
\begin{equation}
\chi(\overline{M^+})=\chi(M)-\chi(\overline{M^-})+\chi(\Sigma)=\chi(M)-\chi(\overline{M^-})+\frac{\#(\Sigma\cap\partial M)}{2}.\label{4.6}
\end{equation}
Hence, by substituting $(\ref{4.6})$ into $(\ref{4.5})$, we obtain
\begin{eqnarray*}
\mathrm{deg}(f)\chi(N)&=&\chi(M)-2\chi(\overline{M^-})+\#S^+-\#S^-\nonumber\\
&&+\#(\Sigma\cap\partial M)^++\frac{\#(\Sigma\cap\partial M)^{\mathrm{null}}}{2}.\label{4.7}
\end{eqnarray*}
Thus, we have the following assertion.

\begin{theorem}
\label{thm4.1}

{\it Let $M,N$ be compact oriented surfaces with boundary and $f:M\to N$ a $C^\infty$-map such that $f^{-1}(\partial N)=\partial M$ and $f$ is transverse to $\partial N$. We suppose that $f$ admits only singular points of the first kind and admissible singular points of the second kind, and that the set of the singular points $\Sigma$ is transversal to the boundary $\partial{M}$. Then, the following equation holds:
\begin{eqnarray*}
\mathrm{deg}(f)\chi(N)&=&\chi(M)-2\chi(\overline{M^-})+\#S^+-\#S^-\nonumber\\
&&+\#(\Sigma\cap\partial M)^++\frac{\#(\Sigma\cap\partial M)^{\mathrm{null}}}{2}.\label{4.8}
\end{eqnarray*}}
\end{theorem}

\begin{remark}
\label{rem4.2}

Theorem $\ref{thm4.1}$ is a generalization of the Quine-Fukuda-Ishikawa formula in \cite[Theorem $1.1$]{3} in the following sense: The formula given in \cite[Theorem $1.1$]{3} is derived under the assumption that a $C^\infty$-map between surfaces admits only fold singular points and cusp singular points. Fold singular points and cusp singular points are typical examples of singular points of the first kind and admissible singular points of the second kind, respectively (see Example $\ref{ex2.8}$).
\end{remark}

\end{document}